\documentclass[10pt,draft,reqno]{amsart}

\makeatletter

\def\section{\@startsection{section}{1}%
	\z@{.7\linespacing\@plus\linespacing}{.5\linespacing}%
	{\bfseries%
		\centering
	}}
	\def\@secnumfont{\bfseries}
	\makeatother

	\newtheorem{theorem}{Theorem}[section]

	\theoremstyle{definition}
	\newtheorem{definition}[theorem]{Definition}
	\newtheorem{example}[theorem]{Example}
	\theoremstyle{remark}
	\newtheorem{remark}[theorem]{Remark}
	\numberwithin{equation}{section}
\AtBeginDocument{{\noindent\small
This is a preprint of a paper whose final and definite form is with 
\emph{Global and Stochastic Analysis} (GSA), ISSN 2248-9444, 
available {\tt http://www.mukpublications.com/gsa.php}. 
First Submitted 26-July-2016; Revised 15-Nov-2016; Accepted 22-Nov-2016.}
\vspace{9mm}}

\begin{document}

\title[A Table of Derivatives on Arbitrary Time Scales]{A Table of Derivatives on Arbitrary Time Scales}

\author{Delfim F. M. Torres}

\address{Center for Research and Development in Mathematics and Applications (CIDMA),
Department of Mathematics, University of Aveiro, 3810--193 Aveiro, Portugal}

\email{delfim@ua.pt}

\urladdr{http://orcid.org/0000-0001-8641-2505}




\begin{abstract}
This paper presents a collection of useful formulas
of dynamic derivatives on time scales, systematically 
collected for reference purposes. As an application,
we define trigonometric and hyperbolic functions 
on time scales in such a way the most essential 
qualitative properties of the corresponding 
continuous functions are generalized in a proper way. 
\end{abstract}

\maketitle


\section{Introduction}

The calculus on time scales is a unification of the theory 
of difference equations with that of differential equations, 
unifying integral and differential calculus with the calculus 
of finite differences and offering a formalism for studying 
hybrid discrete-continuous dynamical systems \cite{MR1062633,MR1066641}.
It has applications in any field that requires simultaneous 
modeling of discrete and continuous data \cite{MR1908825,MR1843232}.
Roughly speaking, the theory is based on a new definition of derivative, 
such that if one differentiates a function which acts on the real numbers, 
then the definition is equivalent to standard differentiation, 
but if one uses a function acting on the integers, then it is equivalent 
to the forward difference operator. The time-scale derivative applies, however, 
not only to the set of real numbers or set of integers, but 
to any time scale, that is, any closed subset $\mathbb{T}$ of the reals, 
such as the Cantor set. It turns out that closed form formulas of derivatives 
of concrete functions, valid on arbitrary time scales, 
are difficult to find in the literature, with some examples scattered
over the literature. This paper intends to fill the gap by gathering
systematically some of the most useful formulas of derivatives on time scales
for reference purposes. Moreover, we present a simple method
on how to prove all such formulas, based on the graininess function 
of the time scale and on a suitable integral representation of the time-scale
derivative (Theorem~\ref{thm:mt}).

The text is organized as follows. Section~\ref{sec:2} 
presents the notations used along the manuscript and 
recalls the definition of delta derivative on time scales and one
of its well-know important properties (Theorem~\ref{thm:util}).
Section~\ref{sec:3} provides a ``table'' of delta derivatives for twenty
functions, organized in different subsections: derivatives of basic functions
(Section~\ref{sec:3:1}); roots (Section~\ref{sec:3:2});
logarithms (Section~\ref{sec:3:3});
exponentials (Section~\ref{sec:3:4});
trigonometric functions (Section~\ref{sec:3:5});
products of trigonometric functions and monomials (Section~\ref{sec:3:6})
and trigonometric functions and exponentials (Section~\ref{sec:3:7});
and hyperbolic functions (Section~\ref{sec:3:8}). In Section~\ref{sec:4}
we prove a simple but powerful integral representation of the
time-scale derivative (Theorem~\ref{thm:mt}) that allows to
obtain all given relations at Section~\ref{sec:3} and others
the reader may wish to prove. We end with Section~\ref{sec:5} where, 
as an application of the given table of derivatives, we propose new definitions 
of trigonometric and hyperbolic functions on time scales
that extend, in a proper way, the most essential qualitative 
properties of the corresponding continuous functions. 

Two remarks are in order. Similarly to the classical calculus,
integration on time scales is the inverse operation of differentiation,
that is, the delta integral is defined as the antiderivative 
with respect to the delta derivative. This means that our formulas 
provide also closed formulas for integration on time scales.
Our second remark concerns the dual concept 
of nabla derivative \cite{MR1911548,MR2771298}:
all formulas given here are also true for the nabla derivative
if one substitutes the forward graininess function by the symmetric 
of the backward graininess function, that is, 
$\mu(t)$ by $- \nu(t)$.


\section{The delta derivative}
\label{sec:2}

A {\it time scale} $\mathbb{T}$ is an arbitrary nonempty closed subset
of $\mathbb{R}$.  Besides standard cases of $\mathbb{R}$ (continuous time)
and $\mathbb{Z}$ (discrete time), many different models of time may be used,
\textrm{e.g.}, the $h$-numbers, 
$\mathbb{T} = h\mathbb{Z}:=\{h z \ | \  z \in \mathbb{Z}\}$,
where $h>0$ is a fixed real number, and the $q$-numbers, 
$\mathbb{T} = q^{\mathbb{N}_0}:=\{q^k \ | \  k \in \mathbb{N}_0\}$,
where $q>1$ is a fixed real number. A time scale $\mathbb{T}$ 
has the topology that it inherits from the real numbers
with the standard topology. For each time scale $\mathbb{T}$, 
the following operators are used:
\begin{itemize}
\item the {\it forward jump operator} $\sigma:\mathbb{T} \rightarrow \mathbb{T}$,
defined by $\sigma(t):=\inf\{s \in \mathbb{T}:s>t\}$ for $t\neq\sup \mathbb{T}$
and $\sigma(\sup\mathbb{T})=\sup\mathbb{T}$ if $\sup\mathbb{T}<+\infty$;
	
\item the {\it backward jump operator} $\rho:\mathbb{T} \rightarrow \mathbb{T}$,
defined by $\rho(t):=\sup\{s \in \mathbb{T}:s<t\}$ for $t\neq\inf \mathbb{T}$
and $\rho(\inf\mathbb{T})=\inf\mathbb{T}$ if $\inf\mathbb{T}>-\infty$;
	
\item the {\it forward graininess function} $\mu:\mathbb{T} \rightarrow [0,\infty)$,
defined by $\mu(t):=\sigma(t)-t$;
	
\item the \emph{backward graininess function}
$\nu:\mathbb{T}\rightarrow[0,\infty)$, defined by
$\nu(t):=t - \rho(t)$.
	
\end{itemize}

\begin{example}
If $\mathbb{T}=\mathbb{R}$, then for any $t \in \mathbb{T}$
one has $\sigma(t)=\rho(t)=t$ and $\mu(t) = \nu(t) \equiv 0$.
If $\mathbb{T}=h\mathbb{Z}$, $h > 0$, then for every $t \in \mathbb{T}$
one has $\sigma(t)=t+h$, $\rho(t)=t-h$, and $\mu(t) = \nu(t) \equiv h$.
On the other hand, if $\mathbb{T} = q^{\mathbb{N}_0}$, $q>1$,
then we have $\sigma(t) = q t$, $\rho(t) = q^{-1} t$,
$\mu(t)= (q-1) t$, and $\nu(t)= (1-q^{-1}) t$.
\end{example}

A point $t\in\mathbb{T}$ is called \emph{right-dense},
\emph{right-scattered}, \emph{left-dense} or \emph{left-scattered},
if $\sigma(t)=t$, $\sigma(t)>t$, $\rho(t)=t$
or $\rho(t)<t$, respectively. We say that $t$ is \emph{isolated}
if $\rho(t)<t<\sigma(t)$, that $t$ is \emph{dense} if $\rho(t)=t=\sigma(t)$.

If $\sup \mathbb{T}$ is finite and left-scattered, we define
$\mathbb{T}^\kappa := \mathbb{T}\setminus \{\sup\mathbb{T}\}$,
otherwise $\mathbb{T}^\kappa :=\mathbb{T}$.

\begin{definition}
\label{def:de:dif} 
We say that a function
$f:\mathbb{T}\rightarrow\mathbb{R}$ is \emph{delta differentiable}
at $t\in\mathbb{T}^{\kappa}$ if there exists a number
$f^{\Delta}(t)$ such that for all $\varepsilon>0$ there is a
neighborhood $U$ of $t$ (\textrm{i.e.},
$U=(t-\delta,t+\delta)\cap\mathbb{T}$ for some $\delta>0$) such that
$$
|f(\sigma(t))-f(s)-f^{\Delta}(t)(\sigma(t)-s)|
\leq\varepsilon|\sigma(t)-s|,\mbox{ for all $s\in U$}.
$$ 
We call $f^{\Delta}(t)$ the \emph{delta derivative} of $f$ at $t$ and $f$ is
said \emph{delta differentiable} on $\mathbb{T}^{\kappa}$ provided
$f^{\Delta}(t)$ exists for all $t\in\mathbb{T}^{\kappa}$.
\end{definition}

\begin{remark}
If $t \in \mathbb{T} \setminus \mathbb{T}^\kappa$, then
$f^{\Delta}(t)$ is not uniquely defined, since for such a point $t$,
small neighborhoods $U$ of $t$ consist only of $t$ and, besides, we
have $\sigma(t) = t$. For this reason, maximal left-scattered points
are omitted in Definition~\ref{def:de:dif}.
\end{remark}

\begin{example}
\label{ex:der:QC}
If $\mathbb{T}=\mathbb{R}$, then $f:\mathbb{R} \rightarrow \mathbb{R}$ is
delta differentiable at $t \in \mathbb{R}$ if and only if $f$ is differentiable in
the ordinary sense at $t$. Then, $f^{\Delta}(t)=f'(t)$.
If $\mathbb{T} = h\mathbb{Z}$, $h > 0$, then
$f:\mathbb{Z} \rightarrow \mathbb{R}$ is always delta differentiable
at every $t \in \mathbb{Z}$ with $f^{\Delta}(t) = \frac{f(t+h) - f(t)}{h}$.
If $\mathbb{T} = q^{\mathbb{N}_0}$, $q>1$,
then $f^{\Delta} (t)=\frac{f(q t)-f(t)}{(q-1) t}$,
\textrm{i.e.}, we get the usual $q$-derivative
of quantum calculus \cite{QC}.
\end{example}

\begin{theorem}[See Theorem~1.16 of \cite{MR1843232}]
\label{thm:util}
Let $\mathbb{T}$ be an arbitrarily given time scale.
If $f$ is delta differentiable at $t\in\mathbb{T}^{\kappa}$, then 
$f(\sigma(t)) = f(t) + \mu(t) f^\Delta(t)$.
\end{theorem}


\section{Table of derivatives on time scales}
\label{sec:3}

In what follows, $\mathbb{T}$ is an arbitrary time scale
with forward graininess function $\mu(t)$; $k$ and $c$ denote arbitrary
constants; while $n$ is a positive integer number: 
$n \in \mathbb{N} = \{1, 2, 3, \ldots\}$. The following formulas hold
for $t \in \mathbb{T}^\kappa$ in the domain of the function 
under consideration. For example, expression \eqref{eq:R01}
for the root function $f(t) = \sqrt{t}$ holds for any 
$t \in \mathbb{T}^\kappa$ such that $t > 0$.
Note that in all the formulas given here, 
if substituting a value into an expression gives $0/0$, 
the expression has an actual finite value
and one should use limits to determine it.
For example, if $\mu(t) = 0$, then formula \eqref{eq:B03}
gives $\ln(k) \, k^t$ and formula \eqref{eq:R01}
gives $\frac{1}{2 \sqrt{t}}$. 


\subsection{Derivatives of basic functions}
\label{sec:3:1}

\begin{equation}
\label{eq:B01}
(k)^\Delta = 0
\end{equation}

\begin{equation}
\label{eq:B02}
\left(t^n\right)^\Delta = \sum_{i=1}^{n} \binom{n}{i-1}\mu(t)^{n-i} \, t^{i-1}
\end{equation}

\begin{equation}
\label{eq:B03}
\text{For } k > 0,\quad 
\left(k^t\right)^\Delta = \frac{k^{\mu(t)}-1}{\mu(t)} \, k^t
\end{equation}

\begin{equation}
\label{eq:B04}
\left((t+k)^n\right)^\Delta 
= \sum_{i=1}^{n} \binom{n}{i-1} \mu(t)^{n-i} (t+k)^{i-1}  
\end{equation}


\subsection{Derivatives of roots}
\label{sec:3:2}

\begin{equation}
\label{eq:R01}
(\sqrt{t})^\Delta = \frac{\sqrt{t+\mu(t)}-\sqrt{t}}{\mu(t)}
\end{equation}

\begin{equation}
\label{eq:R02}
(\sqrt{k+t^n})^\Delta = \frac{\sqrt{k+(t+\mu(t))^n}-\sqrt{k+t^n}}{\mu(t)}
\end{equation}

\begin{equation}
\label{eq:R03}
\left(t^n \sqrt{k+c t}\right)^\Delta 
= \frac{\sqrt{k+c(t+\mu(t))} 
\left(\sum_{i=0}^{n} \binom{n}{i} \mu(t)^{n-i} t^i\right) - \sqrt{k+ct} \, t^n}{\mu(t)}
\end{equation}


\subsection{Derivatives of logarithms}
\label{sec:3:3}

\begin{equation}
\label{eq:L01}
(\ln(t^n))^\Delta 
= \frac{n \ln\left(1+\frac{\mu(t)}{t}\right)}{\mu(t)}
\end{equation}

\begin{equation}
\label{eq:L02}
(\ln(k t + c))^\Delta 
= \frac{\ln\left(1+\frac{k \mu(t)}{k t + c}\right)}{\mu(t)}
\end{equation}


\subsection{Derivatives of exponentials}
\label{sec:3:4}

\begin{equation}
\label{eq:E01}
\left(e^{k t}\right)^\Delta 
= \frac{e^{k \mu(t)}-1}{\mu(t)} \, e^{k t}
\end{equation}

\begin{equation}
\label{eq:E02}
\left(t^n \, e^{k t}\right)^\Delta 
= \frac{\left(\sum_{i=0}^{n} \binom{n}{i} \mu(t)^{n-i} \, 
t^i\right)e^{k \mu(t)}-t^n}{\mu(t)} \, e^{k t}
\end{equation}


\subsection{Derivatives of trigonometric functions}
\label{sec:3:5}

\begin{equation}
\label{eq:T01}
(\sin t)^\Delta 
= \frac{\sin t \, (\cos \mu(t) - 1) + \cos t \, \sin \mu(t)}{\mu(t)}
\end{equation}

\begin{equation}
\label{eq:T02}
(\cos t)^\Delta 
= \frac{\cos t \, (\cos \mu(t) - 1) - \sin t \, \sin \mu(t)}{\mu(t)}
\end{equation}


\subsection{Derivatives of products of trigonometric functions and monomials}
\label{sec:3:6}

\begin{equation}
\label{eq:TM01}
(t \sin(k t))^\Delta 
= \frac{(t+\mu(t))\left(\sin(k t) \cos(k \mu(t))+\cos(k t)\sin(k \mu(t))\right) 
- t \sin(k t)}{\mu(t)}
\end{equation}

\begin{equation}
\label{eq:TM02}
(t \cos(k t))^\Delta 
= \frac{(t+\mu(t))\left(\cos(k t) \cos(k \mu(t))-\sin(k t)\sin(k \mu(t))\right) 
	- t \cos(k t)}{\mu(t)}
\end{equation}


\subsection{Products of trigonometric functions and exponentials}
\label{sec:3:7}

\begin{equation}
\label{eq:TE01}
\left(e^{k t} \sin(ct)\right)^\Delta 
= \frac{\left( \cos(ct) \sin(c \mu(t)) 
+ \sin(ct) \cos(c \mu(t)) \right) e^{k \mu(t)} 
- \sin(c t)}{\mu(t)} \, e^{k t} 
\end{equation}

\begin{equation}
\label{eq:TE02}
\left(e^{k t} \cos(ct)\right)^\Delta 
= \frac{\left( \cos(ct) \cos(c \mu(t)) 
- \sin(ct) \sin(c \mu(t)) \right) e^{k \mu(t)} 
	- \cos(c t)}{\mu(t)} \, e^{k t} 
\end{equation}


\subsection{Derivatives of hyperbolic functions}
\label{sec:3:8}

\begin{equation}
\label{eq:H01}
\left(\sinh(kt)\right)^\Delta 
= \frac{\left(e^{k (t+\mu(t))} - e^{k t}\right)-\left(e^{-k (t+\mu(t))}-e^{-k t}\right)}{2\mu(t)} 
\end{equation}

\begin{equation}
\label{eq:H02}
\left(\cosh(kt)\right)^\Delta 
= \frac{\left(e^{k (t+\mu(t))} - e^{k t}\right)+\left(e^{-k (t+\mu(t))}-e^{-k t}\right)}{2\mu(t)} 
\end{equation}

\begin{equation}
\label{eq:H03}
\left(\sinh(kt) \cosh(kt)\right)^\Delta 
= \frac{\left(e^{2k(t + \mu(t))}-e^{2 k t}\right) - \left(e^{-2k(t + \mu(t))}-e^{-2 k t}\right)}{4\mu(t)} 
\end{equation}


\section{On the proof of equalities (\ref{eq:B01})--(\ref{eq:H03})}
\label{sec:4}

All the equalities \eqref{eq:B01}--\eqref{eq:H03}
can be proved from the following result.

\begin{theorem}
\label{thm:mt}
Let $f : \mathbb{R} \rightarrow \mathbb{R}$ be differentiable.
Let $\mathbb{T}$ be a given time scale with graininess
function $\mu(t)$. If $f$ is delta differentiable 
at $t \in \mathbb{T}^\kappa$, then
\begin{equation}
\label{eq:myder}
f^\Delta(t) = \int_0^1 f'(t + \tau \mu(t)) d\tau.
\end{equation}
\end{theorem}

\begin{proof}
Let $\phi(\tau) = f(t+\tau \mu(t))$.
Then $\phi(0) = f(t)$, $\phi(1) = f(\sigma(t))$,
and $\phi'(\tau) = f'(t+\tau \mu(t)) \mu(t)$.
From the fundamental theorem of calculus, one has
\begin{equation*}
\begin{split}
f(\sigma(t))-f(t) 
&= \phi(1) - \phi(0) = \int_0^1 \phi'(\tau) d\tau\\
&= \mu(t) \int_0^1 f'(t+\tau \mu(t)) d\tau. 
\end{split}
\end{equation*}
Equality \eqref{eq:myder} follows from Theorem~\ref{thm:util}.
\end{proof}


\section{Trigonometric and hyperbolic functions on time scales}
\label{sec:5}

The question of finding appropriate special functions on arbitrary time scales
is an interesting and nontrivial subject under strong investigation: 
see, e.g., \cite{MR2191011,MR2555897}. 
For a survey of existing definitions and some 
new proposals of trigonometric and hyperbolic functions on time scales
see \cite{MR3128436,MR2869725} and references therein. Here, based 
on our formulas of Section~\ref{sec:3}, we propose new versions 
of the sine, cosine, hyperbolic sine and hyperbolic cosine functions 
that extend, in a proper way, the most essential qualitative properties 
of the corresponding continuous functions. 


\subsection{New definitions of trigonometric functions on time scales}
\label{sec:5:1}

Having in mind \eqref{eq:T01} and \eqref{eq:T02} 
and the properties in $\mathbb{R}$
$$
(\sin t)' = \cos t 
\quad \text{and} \quad 
(\cos t)' = -\sin t,
$$
we define the sine and cosine functions on time scales, 
denoted respectively by $\sin_\mathbb{T}(t)$ and $\cos_\mathbb{T}(t)$, 
as follows.

\begin{definition}[Trigonometric functions]
\label{def:trig}
Let $\mathbb{T}$ be an arbitrarily given time scale 
with forward graininess function $\mu(t)$. The 
sine function on the time scale $\mathbb{T}$,
$\sin_\mathbb{T} : \mathbb{T} \rightarrow \mathbb{R}$, is defined by 
\begin{equation}
\label{eq:def:sin}
\sin_\mathbb{T}(t)
= \frac{\sin t \, \sin \mu(t) - \cos t \, (\cos \mu(t) - 1)}{\mu(t)}
\end{equation}
while the cosine function on the time scale $\mathbb{T}$,
$\cos_\mathbb{T} : \mathbb{T} \rightarrow \mathbb{R}$, is defined by
\begin{equation}
\label{eq:def:cos}
\cos_\mathbb{T}(t)
= \frac{\sin t \, (\cos \mu(t) - 1) + \cos t \, \sin \mu(t)}{\mu(t)}.
\end{equation}
\end{definition}

\begin{remark}
\label{rem:sin:cos}
For $\mathbb{T} = \mathbb{R}$ one has $\mu(t) \equiv 0$
and we obtain from \eqref{eq:def:sin} and \eqref{eq:def:cos} that
$$
\sin_\mathbb{R}(t) = \sin(t)
\quad \text{and} \quad
\cos_\mathbb{R}(t) = \cos(t).
$$
\end{remark}

\begin{theorem}
Let $\mathbb{T}$ be an arbitrarily given time scale.
The following relation holds:
\begin{equation}
\label{rel:cos:sin}
\sin_\mathbb{T}^2(t) + \cos_\mathbb{T}^2(t)   
= \frac{2 (1 - \cos \mu(t))}{\mu^{2}(t)}.
\end{equation}
\end{theorem}

\begin{proof}
Follows from Definition~\ref{def:trig} by a direct computation.
\end{proof}

\begin{remark}
For $\mathbb{T} = \mathbb{R}$ one obtains from Remark~\ref{rem:sin:cos}
and \eqref{rel:cos:sin} the fundamental Pythagorean trigonometric identity
of the continuous case: $\sin^2(t) + \cos^2(t) = 1$.
\end{remark}

\subsection{New definitions of hyperbolic functions on time scales}
\label{sec:5:2}

Based on \eqref{eq:H01} and \eqref{eq:H02}, 
and having in mind the properties  in $\mathbb{R}$
$$
(\sinh(t))' = \cosh(t)
\quad \text{and} \quad 
(\cosh(t))' = \sinh(t),
$$
we now define suitable hyperbolic functions sine and cosine on time scales, 
denoted respectively by $\sinh_\mathbb{T}(t)$ and $\cosh_\mathbb{T}(t)$.

\begin{definition}[Hyperbolic functions]
\label{def:hf}
Let $\mathbb{T}$ be an arbitrarily given time scale 
with graininess function $\mu(t)$. The 
hyperbolic sine on the time scale $\mathbb{T}$
is the function $\sinh_\mathbb{T} : \mathbb{T} \rightarrow \mathbb{R}$
defined by 
\begin{equation}
\label{eq:def:sinh}
\sinh_\mathbb{T}(t) 
= \frac{\left(e^{t+\mu(t)} - e^{t}\right)+\left(e^{-(t+\mu(t))}-e^{-t}\right)}{2\mu(t)}. 
\end{equation}
The hyperbolic cosine on the time scale $\mathbb{T}$
is the function $\cosh_\mathbb{T} : \mathbb{T} \rightarrow \mathbb{R}$ given by
\begin{equation}
\label{eq:def:cosh}
\cosh_\mathbb{T}(t) = \frac{\left(e^{t+\mu(t)} - e^{t}\right)
-\left(e^{-(t+\mu(t))}-e^{-t}\right)}{2\mu(t)}. 
\end{equation}
\end{definition}

\begin{remark}
\label{rem:sinh:cosh}
For $\mathbb{T} = \mathbb{R}$ one has $\mu(t) \equiv 0$
and we obtain from \eqref{eq:def:sinh} and \eqref{eq:def:cosh}
the classical definitions:
$$
\sinh_\mathbb{R}(t) = \frac{e^{2t} - 1}{2 e^t} = \sinh(t),
\quad \cosh_\mathbb{R}(t) = \frac{e^{2t} + 1}{2 e^t} = \cosh(t).
$$
\end{remark}

\begin{theorem}
Let $\mathbb{T}$ be an arbitrarily given time scale.
The following relation holds:
\begin{equation}
\label{rel:cosh:sinh}
\cosh_\mathbb{T}^2(t) - \sinh_\mathbb{T}^2(t) 
= \frac{e^{\mu(t)} +e^{-\mu(t)}-2}{\mu^{2}(t)}.
\end{equation}
\end{theorem}

\begin{proof}
Follows from Definition~\ref{def:hf} by a direct computation.
\end{proof}

\begin{remark}
For $\mathbb{T} = \mathbb{R}$ one obtains from Remark~\ref{rem:sinh:cosh}
and \eqref{rel:cosh:sinh} the most important qualitative property 
of the hyperbolic functions in the continuous case: 
$\cosh^2(t) - \sinh^2(t) = 1$.
\end{remark}


\section*{Acknowledgements}

This work was partially supported by Portuguese funds through CIDMA
and FCT, within project UID/MAT/04106/2013. 



\end{document}